\documentclass[12pt]{amsart}

\usepackage{amscd}

\advance\hoffset by -0.5 truecm

\usepackage{amssymb}

\newtheorem{Theorem}{Theorem}[section]

\newtheorem{Lemma}[Theorem]{Lemma}

\newtheorem{Corollary}[Theorem]{Corollary}

\advance\hoffset by -0.5 truecm

\usepackage{amssymb}
\usepackage{graphicx}
\usepackage{subfigure}

\newtheorem{Proposition}[Theorem]{Proposition}

\newtheorem{Definition}[Theorem]{Definition}

\def\V{\mbox{Var}}

\def\R\re

\def\V{\bf V}

\def \re{{\mathbb R}}

\def \0{\lambda_{0}}

\begin{document}
\hyphenation{Ya-ma-be}
\title{Stable solutions of the Yamabe equation on non-compact manifolds}

\author[J. Petean]{Jimmy Petean}

\address{CIMAT, A.P. 402, 36000, Guanajuato. Gto., M\'{e}xico.}
\email{jimmy@cimat.mx}

\author[J. M. Ruiz]{Juan Miguel Ruiz}
 \address{ENES UNAM \\
           37684 \\
          Le\'on. Gto. \\
          M\'exico.}
\email{mruiz@enes.unam.mx}

\thanks{The authors are supported  supported by grant 220074 of CONACYT}



\begin{abstract}  We consider the Yamabe equation on a complete non-compact Riemannian  manifold
and study the condition of stability of solutions. If $(M^m ,g)$ is a closed manifold of constant positive scalar curvature, which we normalize to be $m(m-1)$, we consider the
Riemannian product with the $n$-dimensional Euclidean space: $(M^m \times \re^n , g+ g_E )$. And study, as in \cite{Akutagawa}, the solution of the Yamabe equation 
which depends only on the Euclidean factor. We show that there exists a constant $\lambda (m,n)$
such that the solution is stable if and only if $\lambda_1 \geq \lambda (m,n)$, where $\lambda_1$ is the first positive eigenvalue  of $-\Delta_g$. We compute
$\lambda (m,n)$ numerically for small values of $m,n$ showing in these cases that the Euclidean minimizer is stable in the case $M=S^m $ with the metric of constant curvature. 
This implies that the same
is true for any closed manifold with a Yamabe metric.
\end{abstract}

\maketitle

\section{Introduction}

Let $(X^N ,h)$ be a complete non-compact  Riemannian manifold of dimension $N\geq 3$, without boundary. We consider the $h$-Yamabe functional given by:

$$Y_h (u) = \frac{\int_X \ \left(   a_N \| \nabla u \|^2 + s_h u^2 \right) \  dv_h }{ (\int_X u^p dv_h )^{2/p}} = \frac{E_h (u)}{ \| u \|_p^2} .$$

\noindent
where $a_N = \frac{4(N-1)}{N-2}$, $p=p_N = \frac{2N}{N-2}$, $s_h$ will denote the scalar curvature of the metric $h$ and $dv_h$ its volume element.
The function $u \neq 0$ is assumed to be in the Sobolev space  $L_1^2 (X)$. We will always assume that $(X,h)$ is such that the Sobolev embedding
 $L_1^2 (X) \subset L^p (X)$ holds. This is true for instance if the injectivity radius is positive and the Ricci curvature is bounded below \cite[Corollary 3.19]{Hebey}.

 The Yamabe constant of $(X,h)$ is defined as

$$Y(X,h)= \inf_{u\in L_1^{2} (X) - \{ 0 \} } Y_h (u) .$$

When $s_g \geq 0$ this number is always finite (and non-negative) and it is bounded above by the Yamabe constant of $(S^N , g_0^N )$, where $g_0^N$ is
the metric of constant sectional curvature 1 on $S^N$, by the well known local argument of T. Aubin \cite{Aubin}. 

Although Yamabe constants have been more often considered and are better understood in the case of
closed manifolds, the study of the constants for open Riemannian manifolds is also of interest by itself and in connection with  the closed case. A general study of Yamabe constants
of noncompact manifolds can be found in \cite{Grose}. See also \cite{Botvinnik, Akutagawa, Schoen}

Our main motivation is to understand the Yamabe constants of certain non-compact Riemannian manifolds which play a central 
role in the study of the Yamabe invariants of closed manifolds (in particular when studying how the invariants behave under surgery,
see \cite{Ammann}). In the present article we will consider the stability of solutions of the
Yamabe equation. A solution $f$ of the $h$-Yamabe equation is a solution of the Euler-Lagrange equation of $Y_h$ which means
that for each $u \in  C_0^{\infty} (X) $ the function $H_u (t) = Y_h (f + tu)$ verifies 
$H_u ' (0) =0$. The solution $f$ is called stable if for every $u$, $H_u '' (0) \geq 0$. The condition is well understood in the closed case:
$f$ being a solution of the Yamabe equation means that $f^{p-2} h$ has constant scalar curvature and it is stable if and only if
$s_{f^{p-2} h}  \leq (N-1)  \lambda_1 (f^{p-2} h )$, where $\lambda_1$ is the first positive eigenvalue of
the positive Laplacian of the Riemannian metric. This condition can be expressed also in terms of the original metric $h$, but in the closed 
case there is no reason to use such expression. A typical situation of interest in the complete noncompact case is a  metric of constant positive 
scalar curvature and infinite volume for which one is interested in computing the Yamabe constant. A solution of the Yamabe equation 
gives a metric of constant scalar curvature which 
is non-complete, of finite volume. Since the analysis in such a manifold is not well understood it seems more reasonable to work on the original
metric.
Therefore we will begin this article by studying the stability condition on a non-compact complete Riemannian manifold of constant positive scalar curvature. 

We introduce the following invariant:

\begin{Definition} Let $(X,h)$ be a complete Riemannian manifold of constant positive scalar curvature and 
$f\in C_+^{\infty} (X) \cap L_1^2 (X)$ be a positive smooth critical point of $Y_h$. 
Let $N(h,f) =  \{u \in L_1^2 (X) - \{ 0 \}: \int_X  f^{p-1} u \  dv_h =0 \}$ and  define

$$\alpha (X,h,f) = \inf_{u\in N(h,f) } \ \frac{E_h (u)}{ \int_X   f^{p-2} u^2 dv_h } .$$

\end{Definition}

With this notation the condition for stability reads:

\begin{Theorem}  Let $(X,h)$ be a complete Riemannian manifold of constant positive scalar curvature and $f\in C_+^{\infty} (X) \cap L_1^2 (X)$ 
be a positive smooth critical point of $Y_h$. 
$f$ is stable if and only if 

$$\alpha (X,h,f) \geq  (p-1) \frac{E_h (f)}{\| f \|_p^p} .$$

\end{Theorem}










To study stability of solutions of the Yamabe equation on open manifolds one would need to compute the invariant $\alpha$. 

The example we will be most interested in is  the case $(M^m \times \re^n , g + g_E^n )$ where $M^m $ is closed and $g$ has constant scalar curvature which we
normalize to be $m(m-1)$. One can restrict the functional to functions which depend only on the Euclidean variable and define as
in \cite{Akutagawa}

$$Y_{\re^n} (M\times \re^n , g + g_E^n) = \inf_{u \in L_1^2 (\re^n  ) - \{ 0 \} } Y_{ g + g_E^n  } (u) .$$

In \cite{Akutagawa} $Y_{\re^n} (M\times \re^n , g + g_E^n)$ is computed in terms of the best constant of the classical
Gagliardo-Nirenberg inequality. In particular there is a unique   $Y_{\re^n}$-minimizer $f$ which is a radial, decreasing, smooth 
function and the scalar curvature of $f^{p-2} (g + g_E^n )$ is $m(m-1)$. It follows that $\frac{E_h (f)}{\| f \|_p^p} =m(m-1)$. In
Section 2 we will show that the there is a  minimizer for $\alpha (M\times \re^n , g + g_E^n , f)$ and then in Section 4
we will show that it  is of the form $a(y) b(x)$ where $-\Delta_g a = \lambda_1 a$, $\lambda_1$ is the first positive eigenvalue
of $-\Delta_g$. Then we see from Theorem 1.2 that :

\begin{Theorem} Let $(M^m ,g)$ be a closed Riemannian manifold of constant scalar curvature 
$m(m-1)$ and $f$ the $Y_{\re^n}$-minimizer normalized so that the scalar curvature of $f^{p-2} (g + g_E^n )$ is $m(m-1)$.
$f$ is a stable critical point of $Y_{ g + g_E^n}$ if and only if

\begin{equation}
\inf_{b\in L_1^2 (\re^n ) -\{ 0 \} } \ \left( \frac{ \int_{\re^n} (a_N \| \nabla b \|_2^2 + m(m-1) b^2)}{\int_{\re^n} f^{p-2}b^2} + a_N \lambda_1 \frac{\int_{\re^n} b^2}{\int_{\re^n} f^{p-2}b^2} 
\right) \  \geq (p-1) m (m-1)
\end{equation}

\end{Theorem}

In order to use the previous theorem we will consider the function:

\begin{equation}
\lambda \mapsto A(\lambda )=  \inf_{b\in L_1^2 (\re^n ) -\{ 0 \} } \ \left( \frac{ \int_{\re^n} (a_N \| \nabla b \|_2^2 + m(m-1) b^2)}{\int_{\re^n} f^{p-2}b^2} + a_N \lambda \frac{\int_{\re^n} b^2}{\int_{\re^n} f^{p-2}b^2} \right)
\end{equation}

In section 4 we will prove that $A(\lambda )$  is realized by a radial decreasing function and then deduce the following:

\begin{Corollary} The infimum is a strictly increasing function of $\lambda$. Therefore there exists a unique value
of $\lambda >0$ such that 
$A(\lambda ) = (p-1)m(m-1) .$

\end{Corollary}

We introduce the following constant which depends only on the dimensions $m,n$:

\begin{Definition} The value of $\lambda$ given by the previous corollary will be called $\lambda (m,n)$.
\end{Definition}

 We have

\begin{Theorem} Let $(M,g)$ be a closed Riemannian manifold of constant scalar curvature $m(m-1)$. Let 
$\lambda_1 >0$ be the first positive eigenvalue of $-\Delta_g$. Then the metric $f^{p-2} (g+g_E^n )$ is stable if and only if $\lambda_1 \geq \lambda (m,n)$.
\end{Theorem}

Note that if $g$ is a Yamabe metric (a minimizer for the Yamabe functional) then in particular it is stable and as we mentioned before this
means that $\lambda_1 (g) \geq m$. Therefore we have

\begin{Theorem} If $m \geq \lambda (m,n) $ then for any Yamabe metric $g$ on the closed manifold $M$ the $Y_{\re^n}$-minimizer
on $ (M\times \re^n , g + g_E^n)$ is stable.
\end{Theorem}

The condition on Theorem 1.7 can be checked numerically: a radial mimimizer for $A(\lambda (m,n) )$ is given by a solution of the
ordinary linear differential equation :

\begin{equation} u''(t) + \frac{n-1}{t} u'(t) +\left(  \frac{(p-1)m(m-1)}{a_N}  f^{p-2} - \left( \frac{m(m-1)}{a_N} + \lambda (m,n) \right) \right) u(t)=0
\end{equation}

\noindent
with $u(0)=1$, $u'(0)=0$. In the previous equation replace $\lambda (m,n)$ by a variable $\lambda$. 
As explained in Section 4 using Sturm comparison theory one can easily check that $\lambda (m,n)$ is
the unique value of $\lambda$ such that the solution of previous equation (with the given initial conditions) is positive and decreasing. 
For $\lambda >  \lambda (m,n)$ the solution has a local minimum and for $\lambda <  \lambda (m,n)$ has a 0 at finite time. The function $f$ 
can be computed numerically (see for instance the discussion in \cite{Akutagawa}) and then for a fixed $\lambda$ one can compute numerically
the solution of (3) and check whether  $\lambda <  \lambda (m,n)$ or $\lambda >  \lambda (m,n)$. 

In figure (\ref{K22}) we show the solutions of equation (3) for $m,n=2$. In this case one computes $\lambda (2,2) \approx 1.80405...$ and we display
solutions with $\lambda > \lambda (2,2)$ and $\lambda < \lambda (2,2)$.

Table \ref{tabla} gives the numerical computed value of $\lambda_{m,n}$, for low dimensions ($m+n\leq 9$): in these cases one has
$\lambda_{m,n} \leq m$.

\begin{figure}[h!]
\subfigure[ $\lambda< 1.80405$,  $u (t) =0$ for some $t>0$.]{
		\includegraphics[scale=0.200]{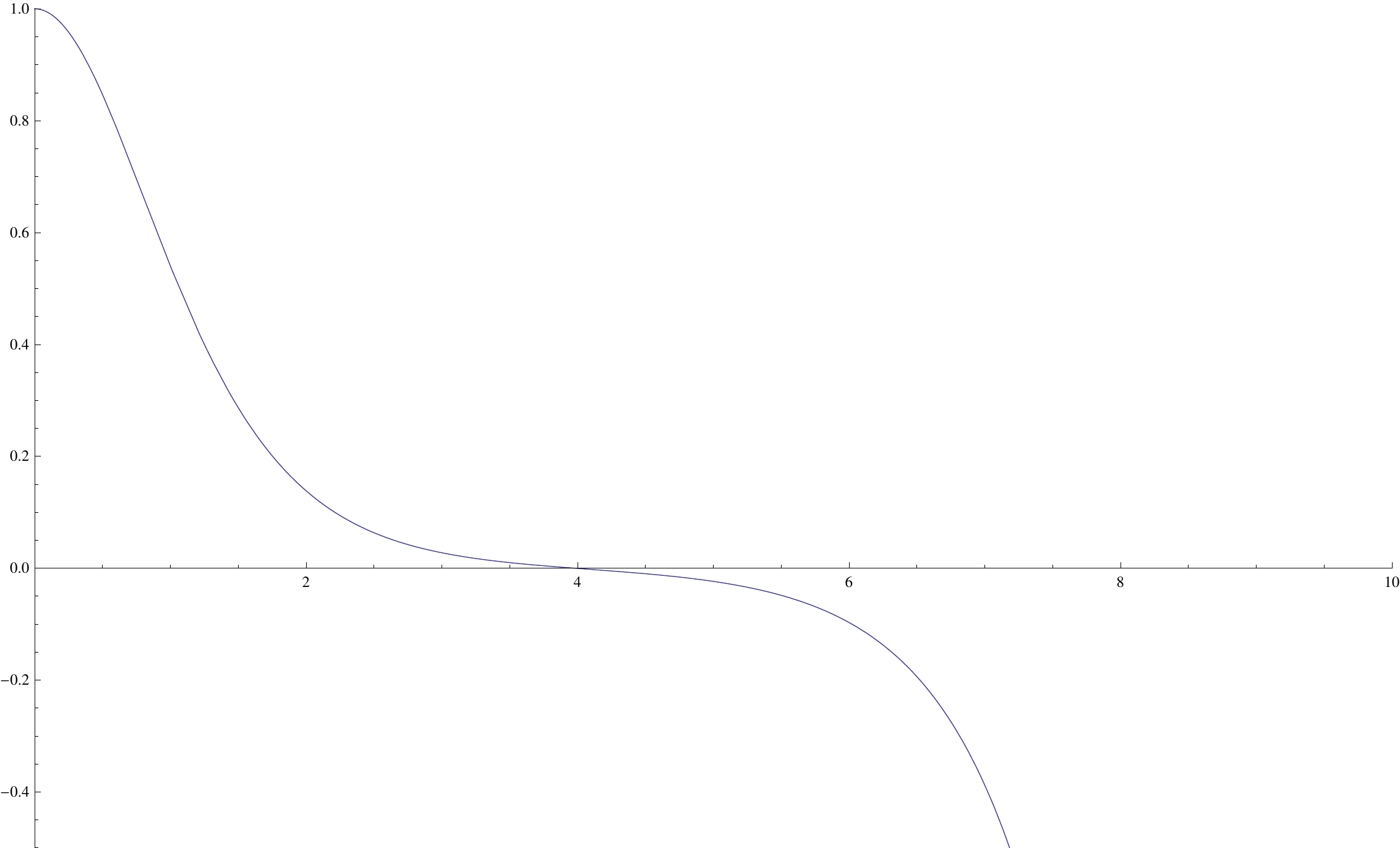}}
\hspace{5pt}	\subfigure[ $\lambda_{2,2}\approx 1.80405$, $u$ is always decreasing.
	]{
		\includegraphics[scale=0.200]{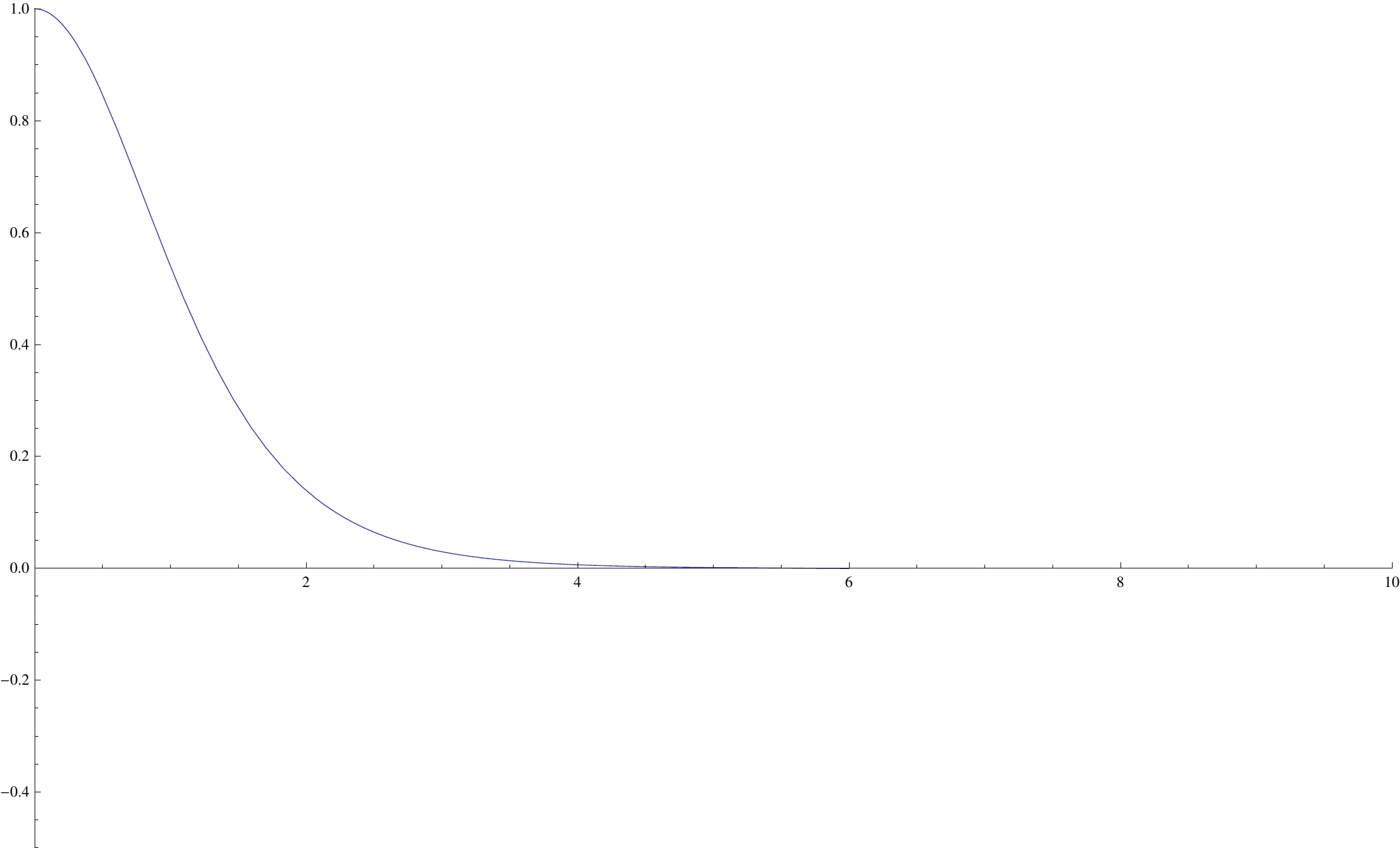}}	
	\hspace{5pt} \subfigure[ $\lambda> 1.80405$, $u $ has a local minimum at some $t > 0$.]{
		\includegraphics[scale=0.200]{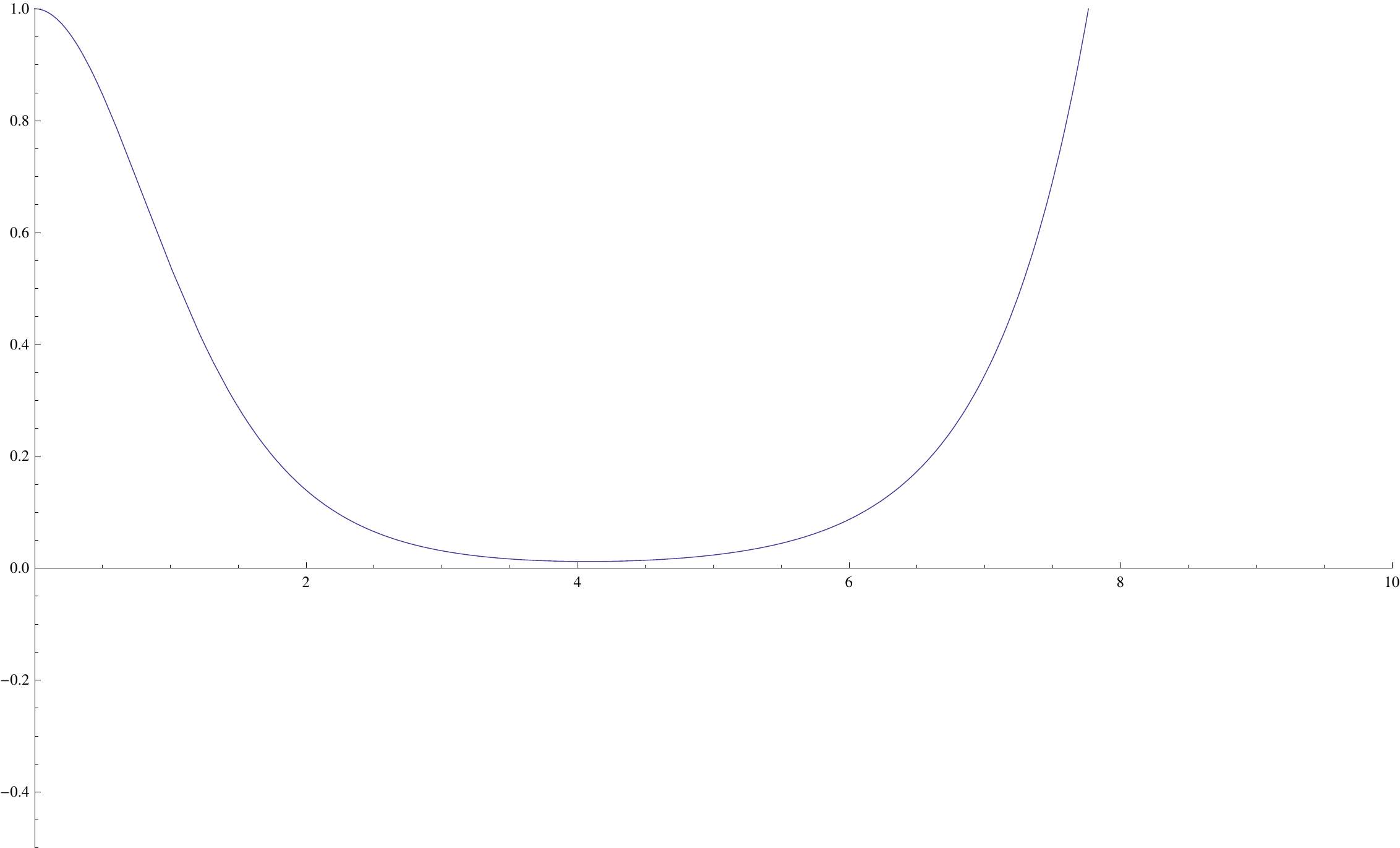}}

	\label{K22}
	\caption{For dimensions $m, n=2$, we display numerical solutions of equation (3). $\lambda_{2,2}\approx 1.80405$ .}
\end{figure}

\begin{table}[!ht]
 \caption {Numerical values of $\lambda_{m,n}$} 
 \label{tabla} 
 
  \centering
  
  \begin{tabular}{ccc}
   
 m	&	n		&$\lambda_{m,n}$	\\
 \hline
 2	&	2	&		1.8041	\\
 3	&	2	&	2.9183	\\
 2	&	3	&	1.6735	\\
 2	&	4	&		1.5823	\\
 3	&	3	&	2.8372	\\
 4	&	2	&		3.9553\\
 2	&	5	&	1.5145	\\
 \hline

  \end{tabular}
  \hspace{1em}
  \begin{tabular}{ccc}
 m	&	n		&$\lambda_{m,n}$	\\
 \hline
3	&	4	&	2.7669	\\
 4	&	3	&		3.9023	\\
 5	&	2	&	4.9718	\\
 2	&	6	&		1.4459	\\
 3	&	5	&	2.7070	\\
 4	&	4	&		3.8506	\\
 5	&	3	&		4.9348	\\
 \hline
  \end{tabular}
  \hspace{1em}
  \begin{tabular}{ccc}
  m	&	n		&$\lambda_{m,n}$	\\
 \hline
  6	&	2	&	5.9806	\\
  2	&	7	&	1.4165	\\
  3	&	6	&	2.6551	\\
  4	&	5	&		3.8028	\\
  5	&	4	&		4.8958	\\
  6	&	3	&	5.9533	\\
  7	&	2	&	6.9859	\\
   \hline
  \end{tabular}
\end{table}

 \textit{Acknowledgement:} The authors would like to thank Prof. Kazuo Akutagawa for many helpful comments on
the first version of the article. 
\section{Yamabe constants of open manifolds}

In this section we will discuss some preliminary definitions and results about Yamabe constants on open manifolds.
For an open  Riemannian manifold $(X^N ,h)$ we consider the $h$-Yamabe functional defined as

$$Y_h (u) = \frac{\int_X \ ( a_N \| \nabla u \|^2 + s_h u^2)  \  dv_h}{ (\int_X u^p dv_h )^{2/p}} $$ 

\noindent
where the function $u$ is taken to be (non-zero) in $L_1^2 (X)$ and recall that we are assuming that the 
Sobolev embedding $L_1^2 \subset L^p$ holds. The Yamabe constant
of $(X,h)$  is then defined as

$$Y(X,h) = \inf_u Y_h (u) $$

Let $E_h (u) = \int_X \ ( a_N \| \nabla u \|^2 + s_h u^2) \  dv_h$ so that $Y_h (u) = E_h (u) \|  u \|_p^{-2} $. 

A critical point of $Y_h$ is a solution of the corresponding Euler-Lagrange which is called the Yamabe equation:

\begin{equation}
- a_N \Delta_h f + s_h f = \lambda  f^{p-1}
\end{equation}

\noindent
with $\lambda \in \re$.

We begin now studying the {\it stability} of solutions of the Yamabe equation. The following is a standard computation:

\begin{Lemma} Let $(X,h)$ be an open manifold and $f$ be a smooth positive  critical point of $Y_h$. For any 
$u\in C_0^{\infty} (X)$ let $H_u (t) = Y_h (f+tu)$. Then $H_u '(0)=0$ and 

$$\frac{H'' (0)}{2} = \frac{ E_h (u) }{\| f \|_p^2}  - \frac{E_h (f)}{\|  f \|_p^4}
\left( (2-p) \| f \|_p^{2-2p} \left( \int_X f^{p-1}u  \right)^2 + (p-1) \| f \|_p^{2-p} \int_X f^{p-2}u^2  \right)  .$$

\end{Lemma}

 \begin{proof} By a standard computation

$$ H'(t)= 2 \frac{ \left(   \int_X \ ( a_N (h(\nabla f ,\nabla u )+ t \| \nabla u \|^2 ) + s_h (fu + tu^2 ))  \ dv_h  \right) \times {\| f+tu \|}_p^2}{ {\| f+tu \|}_p^4 } $$

$$ -2
\frac{ (\int_X \  (a_N \| \nabla (f+tu) \|^2 + s_h (f+t u)^2) dv_h ) \times {\| f+tu \|}_p^{2-p} \times \int_X  (f+tu)^{p-1} u  \ dv_h }{ {\| f+tu \|}_p^4 } .
$$

$H'(0)=0$ since $f$ in a critical point and then by a direct computation

$$H''(0)= \left( \frac{ 2}{ {\| f \|}_p^4 }   \int_X \ ( a_N   \| \nabla u \|^2 + s_h u^2 ) dv_h \right)   \times {\| f \|}_p^2  $$

$$- 2 \frac{E_h(f)}{ \| f \|_p^4}    \left( \frac{2}{p} -1 \right) {\left( \int_X f^p dv_h  \right)}^{2/p -2}  p {\left( \int_X  f^{p-1} u dv_h \right)}^2 $$

$$-2 \frac{E_h (f)}{\| f \|_p^4} (p-1) {\| f \|}_p^{2-p} \int_X f^{p-2} u^2 dv_h      $$

\end{proof}

\begin{Definition} A critical point of the Yamabe functional $Y_h$ is called {\it stable}  if for
each $u\in C_0^{\infty} (M)$ one has $H_u ''(0) \geq 0$.
\end{Definition}

Of course  local  minimizers are stable critical points of $Y_h$. 

The previous lemma now reads:

\begin{Corollary} $f$ is a stable critical point of $Y_h$ if and only if for any  $u\in C_0^{\infty} (X)$

$$  E_h (u)   \geq  E_h (f)
\left( (2-p) \| f \|_p^{-2p} \left( \int_X f^{p-1}u  \right)^2 + (p-1) \| f \|_p^{-p} \int_X f^{p-2}u^2  \right)  .$$

\end{Corollary}

Note that equality holds for $u=f$ since in that case $H_f$ is actually a constant function. Usually one restricts $Y_h$ to metrics
of some fixed volume. In terms of the function $u$ this means that we would consider $u$ such that $\int_X f^{p-1} u =0$. In this situation one 
would have:

\begin{Corollary} A critical point  $f$  of $Y_h$ is stable iff for all $u\in L_1^2 (X)$ such that $\int_X f^{p-1} u dv_h =0$  one has
$E_h (u) \geq (p-1) E_h (f) \| f \|_p^{-p} \int_X f^{p-2} u^2 dv_h $. 
\end{Corollary}

\begin{proof} It is clear that if $f$ is stable then one has the required inequality. Now assume that 
the inequality is true for each $u\in L_1^2 (X)$ such that $\int_X f^{p-1} u dv_h =0$.
 Each $v \in  L_1^2 (X)$ can be written as $v=u+ cf$ where  $u\in L_1^2 (X)$ verifies that  $\int_X f^{p-1} u dv_h =0$ 
and $c\in \re$. Note that then $c= \| f \|_p^{-p} \int_X  f^{p-1} v dv_h$.

Then 

$$E(v) = \int_X \ ( a_N \| \nabla (u+cf) \|^2 + s_h (u+cf)^2)  \  dv_h $$

$$=  \int_X \  a_N \| \nabla u \|^2  -2a_N c u \Delta f + a_N c^2 \| \nabla f \|^2+ s_h u^2 +2cs_h uf + s_h c^2  f^2 \  dv_h $$

$$=E(u) + c^2 E(f)$$

\noindent
(using for the last equality that $-a_N \Delta f + s_h f = \lambda  f^{p-1}$). Then

$$E(v) \| f \|_p^p = ( E(u) + c^2 E(f)  ) \| f \|_p^p  \geq E(f)   (p-1)   \int_X f^{p-2}u^2 dv_h + c^2 E(f) \| f \|_p^p$$

$$ = (p-1) E(f)   \int_X f^{p-2}(v-cf)^2 dv_h + c^2 E(f) \| f \|_p^p .$$ 

$$ =  (p-1) E(f)   \int_X f^{p-2}v^2 dv_h  -2c  (p-1) E(f)  \int_X f^{p-1} v dv_h+ p  c^2 E(f) \| f \|_p^p .$$

\noindent
And replacing the value of $c$ we obtain:

$$E(v) \| f \|_p^2 \geq     (p-1) E(f)   \int_X f^{p-2}v^2 dv_h + E(f) \| f \|_p^{-p}   \left(  \int_X f^{p-1} v dv_h \right)^2 (2-p) $$

This shows that $f$ is a stable critical point.

\end{proof}

Given   a complete Riemannian manifold $(X,h)$ and  $f\in C_+^{\infty} (X) \cap L_1^2 (X)$ a positive smooth critical point of $Y_h$ we let
as in the introduction $N(h,f) =  \{u \in L_1^2 (X) - \{ 0 \}: \int_X f^{p-1} u dv_h =0 \}$ and  call

$$\alpha (X,h,f) = \inf_{u\in N(h,f) } \ \frac{E_h (u)}{ \int_X f^{p-2} u^2 dv_h } .$$

With this notation we have that $f$ is a stable solution of the Yamabe equation if and only if

$$\alpha (X,h,f) \geq (p-1) \frac{E_h (f)}{ \|  f \|_p^p } $$

\noindent
as claimed in Theorem 1.2.

In the next sections we will consider the particular case when $(X,h)=(M\times \re^n ,g+g_E^n )$, a Riemannian product of a closed
Riemmanian manifold of constant positive scalar curvature with the Euclidean space, and
$f$ a critical point of $Y_h$ which is a smooth radial decreasing positive function on $\re^n$. We will
use the fact that $\alpha$ is achieved : 

\begin{Proposition} There exists $u\in N(g+ g_E^n ,f)$  which achieves the infimum in the definition of $\alpha (M \times  \re^n ,g+g_E^n ,f) $. Every minimizer
is a  smooth  function which solves the equation

\begin{equation}
\label{2.6}
-a_n \Delta u + (s_g -\alpha  f^{p-2} ) u =0 
\end{equation}

The space of solutions of this equation is finite dimensional.
\end{Proposition}
\begin{proof} Let $\{ u_i \}$ be a minimizing sequence.  We can assume that $\int_X f^{p-2}u_i^2 dv_h =1$ and $u_i \geq 0$. It
follows that $\{ u_i \}$ is a bounded sequence in $L_1^2 (X)$ and therefore (after taking a subsequence) it has a weak limit $u|_K$ in $L_1^2 (K)$, for every compact $K \subset X$, $u|_K\geq 0$. Also, $u_i $ converges to $u|_K$ in $L^2 (K)$, since the Sobolev embedding is compact for $K \subset X$, and by H\"older's inequality.

Consider now compact subsets $K_R= M \times B_R \subset X$ ($B_R\subset \re^n$ a closed ball with radius $R>0$). Since the convergence on $L^2(K_R)$ is strong for each $R$, $K_{R}\subset K_{R'}$ for $R<R'$, and  $X=\cup_i^{\infty}K_i$, then we have a well defined function on all of $X$, $u=\lim_{R\rightarrow \infty} u|_{K_R}$.

Furthermore, on each compact $K_R$

$$\int_{K_R} |\nabla u|^2 dv_h = \lim_{i\rightarrow \infty} \int_{K_R} \langle\nabla u, \nabla u_i\rangle_h dv_h $$
 \noindent and then, by the Cauchy inequality,
$$\int_{K_R} |\nabla u|^2 dv_h \leq \limsup_{i\rightarrow \infty} \int_{K_R}  |\nabla u_i|^2 dv_h $$

Moreover, by the strong convergence on $L^2(K_R)$

$$\int_{K_R} u^2 dv_h = \lim_{i \rightarrow \infty} \int_{K_R}  u_i^2 dv_h. $$

It follows that 

\begin{equation}
\int_{K_R}(a |\nabla u|^2 +s_h u^2)dv_h\leq\limsup_{i\rightarrow \infty} \int_{K_R} (a |\nabla u_i|^2+s_h u_i^2)dv_h 
\nonumber
\end{equation}
\begin{equation}
\label{EE}
\leq \limsup_{i\rightarrow \infty} \int_{X} (a |\nabla u_i|^2+s_h u_i^2)dv_h \leq \limsup_{i\rightarrow \infty} E_h(u_i)= \alpha.
\end{equation} 

\noindent Then, by making $R\rightarrow \infty$, inequality (\ref{EE}) implies that $E_h(u)\leq \alpha$. Since $\alpha$ is an infimum by definition, it remains to show that $\int_X f^{p-2}u^2 dv_h=1$, to prove that $u$ in fact minimizes $\frac{E_h (u)}{\int_X f^{p-2} u^2 dv_h}$.

This follows from the fact that $f$ is radially dependent on $\re^n$  and decreasing. Given $\epsilon>0$, then, for big $R$, we have $f^{p-2}(r)<\epsilon$, for $r>R$. Hence
$$\int_{X \setminus M\times B_r} u_i^2 f^{p-2} dv_h \leq \epsilon \int_{X \setminus M\times B_r} u_i^2 dv_h \leq \epsilon \int_X u_i^2 dv_h\leq C \epsilon,$$
\noindent for some constant $C$ (recall that $\{ u_i \}$ is a bounded sequence in $L_1^2 (X)$). It follows that for every $r>R$

$$ 1 \geq \lim_{i\rightarrow \infty} \int_{M\times B_r} f^{p-2} u_i^2 dv_h \geq 1- C \epsilon,$$
\noindent that is

$$ 1\geq \int_{M\times B_r} f^{p-2} u^2 dv_h \geq 1- C \epsilon.$$

\noindent Finally, by making $r \rightarrow \infty$, we have $\int_{X} f^{p-2} u^2 dv_h=1$. As stated, this proves that $u$ minimizes $\frac{E_h (u)}{\int_X f^{p-2} u^2 dv_h}$.

Of course, this implies that $\forall$  $\varphi \in C_0^{\infty}(X)$, 
$\frac{d}{dt}\left(\frac{E_h(u+t\varphi)}{\int_X f^{p-2}(u+t \varphi)^2 dv_h}\right)\big|_{t=0}=0.$ That is,
$$ \frac{ 2 a_{m+n}\int_X ( \langle\nabla \varphi,\nabla u\rangle_h+2 s_h \varphi u) \, dv_h}{\left(\int_X f^{-2+p} u^2 \, dv_h\right)^{2/p}}-2 \left(\int_X f^{-2+p} \varphi u \, dv_h\right)  \frac{ \int_X \left(a_{m+n} \nabla u+s_h u^2\right) \, dv_h}{\left(\int_X f^{p-2} u^2 \, dv_h\right)^2}=0,$$
\noindent it follows that
$$   a_{m+n}\int_X ( \langle\nabla \varphi,\nabla u\rangle_h+ s_h \varphi u) \, dv_h- \left(\int_X f^{p-2} \varphi u \, dv_h\right)  \frac{E_h(u)}{\int_X f^{p-2} u^2 \, dv_h}=0,$$
\noindent and then
$$  \int_X \varphi \left( - a_{m+n} \Delta u+ s_h  u - \alpha f^{p-2}  u \right)\, dv_h =0,$$
\noindent for every $\varphi \in C_0^{\infty}(X)$. That is, 
$u$ is a weak solution of equation (\ref{2.6}). The fact that $u$ is a smooth function, follows from standard regularity results (see for example Theorem 4.1 in \cite{parker}).

Finally, we remark that the space of solutions is finite dimensional. Suppose it were infinite dimensional, then we would have a sequence $\{u_i\}$ of minimizers, such that  $\int_X f^{p-2}u_i^2 dv_h =1$, $u_i \geq 0$ and $||u_i-u_k||_2>\epsilon$, for every $i,k$, and for some $\epsilon>0$. By applying  the argument of the proof to this sequence, we would have strong $L^2(X)$ convergence of a subsequence of $\{u_i\}$ to some $L^2(X)$ function $u_0$, contradicting the hypothesis that $||u_i-u_k||_2>\epsilon$.

\end{proof}

\section{The $Y_{\re^n}$-minimizers on $(M\times \re^n , g + g_E^n )$}

We consider a closed Riemannian manifold $(M,g)$ of constant positive scalar curvature. We use the notation
$g_E^n$ for the Euclidean metric on $\re^n$. We will assume always that $m, n \geq 2$.

In general if $(Z,G)=(M_1 \times M_2 , g+h )$ is a Riemannian product we consider as in \cite{Akutagawa} the 
restricion of $Y_G$ to functions on one of the variables and let

$$Y_{M_i} (Z,G) = \inf_{u \in L_1^2 (M_i )} Y_G (u) .$$

In \cite[Theorem 1.4]{Akutagawa} it was proved that $Y_{\re^n} (M \times \re^n , g + g_E^n )$ can be computed in
terms of the best constant in the Gagliardo-Nirenberg inequality.  
The Gagliardo-Nirenberg inequality says that there exists a positive constant $\sigma$ such that for all $u\in L_1^2 (\re^n )$

$$\| u \|_{p_{m+n}}^2 \leq \sigma \| \nabla u \|_2^{\frac{2n}{m+n}} \| u \|_2^{\frac{2m}{m+n}} .$$

The best constant is of course the smallest value $\sigma_{m,n}$ that makes the inequality true:

$$\sigma_{m,n} = \left( \inf_{u\in L_1^2 (\re^n ) - \{ 0 \} } \frac{  \| \nabla u \|_2^{\frac{2n}{m+n}} \| u \|_2^{\frac{2m}{m+n}}   }{\| u \|_{p_{m+n}}^2 } \right)^{-1} .$$

The infimum is actually achieved. The minimizer is a solution of the Euler-Lagrange equation of the functional in parenthesis:

\begin{equation}
- n \Delta u + m \frac{ \| \nabla u \|_2^2}{\| u \|_2^2 } u - (m+n) \frac{ \| \nabla u \|_2^2 }{\| u \|_p^p} u^{p-1} =0 .
\end{equation}

By invariance if a function $u$ is a minimizer so is $cu_{\lambda}$ given by $ cu_{\lambda} (x)= cu(\lambda x)$ for any constants $c, \lambda \in\re_{>0}$.
In terms of equation (6) this means that a solution $u$ gives a 2-dimensional family of solutions. By picking $c, \lambda$
appriopriately we can choose the (constant) coefficients appearing in the equation. In particular one would have
a solution of 

\begin{equation}
-\Delta u + u -u^{p-1} =0.
\end{equation}

It is known since the classical work of Gidas-Ni-Nirenberg \cite{Gidas, Ni} that all  solutions of equation (7) which are 
positive and vanish at infinity are radial functions. It is also known the existence of a radial solution \cite{Strauss}. Moreover,
M. K. Kwong \cite{Kwong} proved that such a solution is unique.

In our situation we will prefer to first choose $\lambda$ so that 
$ a_{m+n} m \| \nabla u \|_2^2    = n s_g \| u \|_2^2$ and
then pick $c$ so that $(m+n)a_{m+n} \| \nabla u \|_2 ^2  = s_g  n  \|  u \|_p^p .$ Then the resulting function
$f_K$ satisfies

\begin{equation}
-a_{m+n} \Delta f_K + s_g \  f_K = s_g \  f_K^{p-1} 
\end{equation}

Note that the function $f_K$ depends on $m,n$ and $s_g$. 
The metric $g_K = f_K^{p-2} (g + g_E^n )$ has scalar curvature $s_{g_K} = s_g $. $g_K$ is a non-complete
metric of finite volume. We will denote the function  $f_K$ by $f=f_K^{m,n, s_g}$ (in case it is necessary to make it explicit the dependence on
$m,n,s_g$). Note that we have:

\begin{equation}
 a_{m+n} m  \| \nabla f_K^{m,n,s_g}  \|_2^2    = n s_g  \| f_K^{m,n,s_g}  \|_2^2
\end{equation}

\begin{equation}
(m+n)a_{m+n} \| \nabla f_K^{m,n,s_g} \|_2 ^2  =   n  s_g  \|  f_K^{m,n,s_g} \|_p^p 
\end{equation}

\begin{equation}
(m+n) \|  f_K^{m,n,s_g} \|_2 ^2  = m  \|  f_K^{m,n,s_g} \|_p^p 
\end{equation}

A minimizer for $Y_{\re^n} (M \times \re^n , g + g_E^n )$ must be a solution of (3). And by the previous comments the
solution is unique, so actually the solution $f_K^{m,n,s_g }$ is the unique minimizer for $Y_{\re^n} (M \times \re^n , g + g_E^n )$.
 We have

$$Y_{\re^n} (M \times \re^n , g + g_E^n ) = s_g  Vol(g_K  )^{\frac{2}{m+n}} .$$

\section{Stability of the  $Y_{\re^n}$-minimizers }

Let $g$ be a Riemannian metric  on the closed $m$-manifold $M$ of constant scalar curvature $s_g = m(m-1)$. To simplify we will use the notation $G=g + g_E^n$, $N=m+n$
Let $f: \re^n  \rightarrow \re_{>0}$ be the unique solution of equation (9) discussed in the previous section.

Note that $E_G (f) = m(m-1) \| f \|_p^p$.

\begin{Lemma} $\alpha = \alpha (M\times \re^n , G,f) < (p-1)m(m-1) $ then it is realized by a function $u(y,x) = a(y) b(x)$ where $a: M \rightarrow \re$ ,  $-\Delta_g a = \lambda_1 a$
(where  $\lambda_1$ is  the first positive eigenvalue) and $b \in L_1^2 (\re^n )$ satisfies the equation:

\begin{equation} 
 -a_N \Delta b +   \left( -a_N  \lambda_1  + m(m-1) - \alpha  f^{p-2}  \right)  b =0.
\end{equation}

\end{Lemma}

\begin{proof} By Proposition 2.5 there exists a minimizer and it is a solution of the equation

$$ -a_N \Delta u +  \left(  m(m-1)  - \alpha  f^{p-2}  \right)  u =0 $$

\noindent
(and the space of solutions of the equation is finite dimensional).
Since $f$ depends only on $\re^n$ it follows that if $u$ is a solution of the equation
then $\Delta_{g} u$ is also a solution. Then for each $x\in \re^n$ the function $u(-,x)$ lies in a 
finite dimensional  $\Delta_{g}$-invariant subspace. It follows that there is a finite
number of linearly independent  $\Delta_{g}$- eigenfunctions $a_1 (y),...,a_k (y)$, $\Delta_{g} a_i = \lambda_i a_i$
($\lambda_i \leq 0$),  such that 
$u= \Sigma a_i (y) b_i (x)$ for some functions $b_i  : \re^n \rightarrow \re$. 

But then we have that 

$$\Sigma_{i=1}^k  \left(-a_N ( \lambda_i a_i (y) b_i (x) + a_i (y) \Delta b_i (x) ) +  \    ( m(m-1) - \alpha  f^{p-2} ) \  a_i(y) b_i (x)\right) =0 .$$

But then since the $a_i$ are linearly independent it follows that for each $i$ 

$$ -a_N ( \lambda_i  b_i (x) +  \Delta b_i (x) )+   \left( m(m-1)  - \alpha  f^{p-2}  \right)  b_i (x) =0 .$$ 

So $a_i b_i$ is also a solution for each $i$. We have proved that there is a minimizer of
the form $a(y) b(x)$ with $-\Delta_g a = \lambda a $ for some $\lambda \geq 0$. If $\lambda =0$ we take
$a=1$ and then we must have $\int_{\re^n} bf^{p-1} dx =0$. Since $f$ is a $Y_{\re^n}$-minimizer it is stable when we restrict the
functional to $L_1^2 (\re^n )$. Then restricting the variation to $C_0^{\infty} (\re^n )$ the same inequality as in Corollary 2.3 gives: 

$$ \alpha  (M\times \re^n , G,f) \geq (p-1) \frac{E_G (f)}{\| f \|_p^p} =(p-1) m(m-1)$$

If $\lambda >0$  note that

$$\frac{E_G (ab)}{\int_{\re^n}  f^{p-2} a^2 b^2 } = \frac{ \int_{\re^n} (a_N \| \nabla b \|_2^2 + s_g b^2)}{\int_{\re^n} f^{p-2}b^2} + a_N \lambda \frac{\int_{\re^n}  b^2}{\int_{\re^n} f^{p-2}b^2} .$$

It follows that for the minimizer we must have $\lambda =\lambda_1$ and the lemma follows.

\end{proof}

Therefore $f$ is unstable if and only if 

\begin{equation}
\inf_{b\in L_1^2 (\re^n ) -\{ 0 \} } \ \left( \frac{ \int_{\re^n} (a_N \| \nabla b \|_2^2 + m(m-1) b^2)}{\int_{\re^n} f^{p-2}b^2} + a_N \lambda_1 \frac{\int_{\re^n} b^2}{\int_{\re^n} f^{p-2}b^2} \right) \  < (p-1) m (m-1)
\end{equation}

\noindent
as claimed in Theorem 1.3. 

\begin{Lemma}
For each $\lambda \geq 0$

$$A(\lambda )= \inf_{b\in L_1^2 (\re^n ) -\{ 0 \} } \ \left( \frac{ \int_{\re^n} (a_N \| \nabla b \|_2^2 + s_g b^2)}{\int_{\re^n} f^{p-2}b^2} +  \lambda \frac{\int_{\re^n} b^2}{\int_{\re^n} f^{p-2}b^2}
\right) \   \geq s_g f(0)^{2-p} $$

\noindent
is realized  by a radial decreasing function.
\end{Lemma}

\begin{proof} Given any $b\in L_1^2 (\re^n ) -\{ 0 \} $ let $b^*$ be its radial decreasing rearrangement. Then since $f$ is also radial and decreasing
we obtain from the Hardy-Littlewood inequality that $\int_{\re^n} f^{p-2}b^2 \leq  \int_{\re^n} f^{p-2}{b^*}^2$. And as usual
$\int b^2 = \int {b^*}^2$ and $  \| \nabla b^*  \|_2^2  \leq   \| \nabla b \|_2^2 $. It follows that for the minimization we can consider only
radial decreasing functions. 
Let $b_i$ be a sequence of radial decreasing functions  such that the corresponding quotient converges
to the infimum. We can normalize de sequence so that $\int f^{p-2} b_i^2  =1$. Then $b_i$ is a bounded sequence in
$L_1^2$ which must have a subsequence converging to $b\in L_1^2$. Since the embedding $L_1^2 \subset L^p$ restricted to
radial functions is compact it follows that the sequence converges to $b$ in $L^p$. But then $\int f^{p-2}b_i^2 \rightarrow
\int f^{p-2} b^2$. It follows that $b$ is a minimizer.

\end{proof}

Since the infimum is realized it follows easily that the infimum is a strictly increasing
function of $\lambda$. Setting $b=f$ for $\lambda =0$ we see that in this case the infimum is at most $m(m-1)$ and of course the infimum tends to
$\infty$ as $\lambda \rightarrow \infty$. 

Therefore there exists a unique value
of $\lambda >0$ such that $A(\lambda ) = (p-1)m(m-1)$, as claimed in Corollary 1.4. This value of $\lambda$ was called
$\lambda (m,n)$ in the introduction and Theorem 1.6 follows from the previous comments.

\vspace{1cm}

The value of $\lambda (m,n)$ can be computed numerically, but since the function $f$ (and correspondingly the
best constant in the Gagliardo-Nirenberg inequality) can only be computed numerically it seems that there is little hope to
obtain an explicit computation of it. To carry on the numerical computation we note that the minimizer $b$ is a solution
of 

$$-a_N \Delta b +(m(m-1) + a_N \lambda (m,n) )b = (p-1)m(m-1) f^{p-2} b .$$
 
In general consider the equation 

\begin{equation}
- \Delta b + Kb = Cf^{p-2}b ,
\end{equation}

\noindent
where $C=(p-1)m(m-1)/a_N $ and $K$ is a (variable)  positive constant.  A radial solution is given by a solution of the ordinary
linear differential equation:

\begin{equation} u''(t) + \frac{n-1}{t} u'(t) +(C f^{p-2} -K) u(t)=0
\end{equation}

\noindent
with $u(0)=1$, $u'(0)=0$.

Note that $u''(0)= (1/n)(K-Cf^{p-2} (0))$. We take   $K<Cf^{p-2} (0)$ so that  the
solution $u$ is decreasing close to 0.  We will denote the 
solution $u$ by $u_K$. We have 3 possibilities:

a) $u_K$ is always decreasing and positive.

b) $u_K (t) =0$ for some $t>0$.

c) $u_K $ has a local minimum at some $t\geq t_0$.

It is easy to see that in  case (a) we have $\lim_{t\rightarrow \infty} u_K (t) = 0$.

By Sturm comparison, as stated for instance in \cite[Lemma 1, page 246]{Kwong} or in Ince's book \cite{Ince}, we have that if 
$0< K_1 < K_2$ and $t_0 >0$ is such that $u_{K_1}$ and $u_{K_2}$ are positive on $[0,t_0 )$ then for all $t\in (0,t_0 )$ we
have 

$$\frac{u_{K_1}'}{u_{K_1}}  < \frac{u_{K_2}'}{u_{K_2}} .$$

It follows that if the solution $u_{K_1}$ verifies (c) then the solution $u_{K_2}$ also verifies (c). If $u_{K_2}$
verifies (b) then $u_{K_1}$ also verifies (b). Moreover if $u_{K_2}$ verifies (a) then $u_{K_1}$ verifies (b).

It follows that for $\lambda = \lambda (m,n)$ the equation 

\begin{equation} u''(t) + \frac{n-1}{t} u'(t) +\left(  \frac{(p-1)m(m-1)}{a_N}  f^{p-2} - \left( \frac{m(m-1)}{a_N} + \lambda  \right) \right) u(t)=0
\end{equation}

\noindent
 is positive and decreasing. 
For $\lambda >  \lambda (m,n)$ the solution has a local minimum and for $\lambda <  \lambda (m,n)$ has a 0 at finite time. The function $f$ 
can be computed numerically (see for instance the discussion in \cite{Akutagawa}) and then for a fixed $\lambda$ one can compute numerically
the solution of (16) and check whether  $\lambda <  \lambda (m,n)$ or $\lambda >  \lambda (m,n)$. 
 In this way one can numerically compute $\lambda_{m,n}$ as mentioned in the introduction.


\begin{thebibliography}{aa}


\bibitem{Botvinnik} K. Akutagawa, B. Botvinnik,
{\it Yamabe metrics on cylindrical manifolds},
Geom. Funct. Anal. {\bf 13} (2003), 259-333.

\bibitem{Akutagawa} K. Akutagawa, L. Florit, J. Petean, {\it On Yamabe constants of Riemannian
products}, Comm. Anal. Geom. {\bf 15} (2007), 947-969. 

\bibitem{Ammann} B. Ammann, M. Dahl, E. Humbert, {\it Smooth Yamabe invariant and surgery}, 
J. Differential Geometry {\bf 94} (2013), 1-58. 


\bibitem{Aubin} T. Aubin, {\it Equations differentielles non-lineaires et
probleme de Yamabe concernant la courbure scalaire},
J. Math. Pures Appl. {\bf 55} (1976), 269-296.

\bibitem{Gidas} B. Gidas, W. M.  NI, L. Nirenberg, {\it Symmetry and related properties via the
maximum principle}, Comm. Math. Phys. {\bf 68}  (1979), 209-243.


\bibitem{Ni}  B. Gidas, W. M.  NI, L. Nirenberg, {\it Symmetry of positive solutions of nonlinear
elliptic equations in $\re^n$}, Advances in Math. Studies 7 A (1981), 369-402.




\bibitem{Grose} N. Gro$\beta$e, M. Nardmann, {\it The Yamabe constant of noncompact manifolds}, J. Geom. Anal. \textbf{24(2)} (2014), 1092-1125.


\bibitem{Hebey} E. Hebey, {\it Sobolev spaces on Riemannian manifolds}, Lecture Notes in Mathematics 1635, Springer-Verlag, 
Berlin, 1996. 



\bibitem{Ince} E. L. Ince, {\it Ordinary differential equations}, Dover Publications, New York, 1956.

\bibitem{Kwong} M. K. Kwong, {\it Uniqueness os positive solutions of $\Delta u -u +u^p =0$ in $\re^n$}, Arch. Rational
Mech. Anal. {\bf 105} (1989), 243-266. 



\bibitem{parker} T. H. Parker and J. M.  Lee, {\it The Yamabe Problem} Bull. of the Amer. Math. Soc. {\bf 17}, Number 1, (1987), 37-91.

\bibitem{Strauss}W. A. Strauss, {\it Existence of solitary waves in higher dimensions},
Comm. Math. Phys. {\bf 55} (1977), 149-162. 

\bibitem{Schoen} R. Schoen, S. T. Yau, {\it Conformally flat manifolds, Kleinian groups and scalar curvature}, 
Invent. Math. {\bf 92} (1988), 47-71. 







\end{thebibliography}
\end{document}